\documentclass{article}

\usepackage{amsmath,amsthm,amssymb,eufrak}

\newtheorem{theorem}{Theorem}[section]

\newtheorem{lemma}[theorem]{Lemma}

\newtheorem{remark}[theorem]{Remark}
\newtheorem{definition}[theorem]{Definition}

\newtheorem{assumption}[theorem]{Assumption}

\newcommand{\N}{\mathbb{N}}
\newcommand{\Z}{\mathbb{Z}}

\newcommand{\R}{\mathbb{R}}

\newcommand{\B}{\mathcal{B}}

\newcommand{\X}{\mathcal{X}}
\newcommand{\Y}{\mathcal{Y}}
\newcommand{\M}{\mathcal{M}}

\renewcommand{\P}{\mathbb{P}}
\newcommand{\E}{\mathbb{E}}

\newcommand{\law}[1]{\mathcal{L}(#1)}

\newcommand{\Pas}{\text{a.s.}}

\newcommand{\eps}{\varepsilon}
\newcommand{\la}{\lambda}
\newcommand{\ga}{\gamma}
\newcommand{\ka}{\kappa}
\newcommand{\dtv}{d_{\text{TV}}}

\newcommand{\dint}{\mathrm{d}}  

\begin{document}


	
	\title{Ergodic theorems for queuing systems with dependent inter-arrival times}
		
	\author{Attila Lovas\thanks{Alfr\'ed R\'enyi Institute of Mathematics, Re\'altanoda utca 13-15, 1053 Budapest
	and Budapest University of Technology and Economics, Egry J\'ozsef utca 1, 1111 Budapest, Hungary; lovas@renyi.hu}
	\and Mikl\'os R\'asonyi\thanks{Alfr\'ed R\'enyi Institute of Mathematics, Re\'altanoda utca 13-15, 1053 Budapest, Hungary;
	rasonyi@renyi.hu}}

\maketitle{}

	\begin{abstract}
		
		We study a G/GI/1 single-server queuing model with i.i.d.\ 
		service times that are independent of a stationary process of inter-arrival times.
		We show that the distribution of the waiting time converges to a stationary law 
		as time tends to infinity provided that 
		inter-arrival times satisfy a G\"artner-Ellis type condition. 
		A convergence rate is given and a law of large numbers established. These results provide tools 
		for the statistical analysis of such systems, transcending the standard case 
		with independent inter-arrival times.
	\end{abstract}
		
\textbf{Keywords:} Queuing; G/GI/1 queue; Dependent random variables; Inter-arrival times; Limit theorem; Law of large numbers

\section{Introduction}

At the beginning of the 20th century, Agner Krarup Erlang, a Danish engineer did 
the first steps in a new branch of operations research what we now call queuing theory \cite{erlang1909}. Erlang worked for a telephone company, where he created a mathematical model to determine the minimal number of telephones needed to process a given volume of calls. Nowadays, queuing theory arises in many fields of engineering sciences such as inventory management, logistics, transportation, industrial engineering, optimal service design, telecommunication, etc. Among the most recent and interesting applications, we can mention a new queuing theory approach for cost reduction in product-service design \cite{Nguyen2017}. Such solutions can be helpful for the service designer to find the optimized solution for the designed service. For applications of queueing theory in health care, we refer the reader to \cite{healthcare2013}, and readers with motivation in
communication systems should consult \cite{Zeephongsekul2006}.

In our paper, we consider a single-server queuing system with infinite buffer and first-in, first-out (FIFO) 
service discipline, where customers are numbered by $n\in\N$. 
Furthermore, we assume that there is a two-dimensional sequence of non-negative
random variables $(S_{n},Z_{n})_{n\in\Z}\in\mathbb{R}_{+}^{2}$ such that
the time between the arrival of customers $n+1$ and $n$ is described by $Z_{n+1}$, for each $n\in\N$. 
The service time for customer $n$ is given by $S_n$, for $n\in\N$. Note that we are working with the index set 
$\mathbb{Z}$ instead of $\mathbb{N}$ for the process $(S,Z)$, where the
extension to the negative time axis is for mathematical convenience.

The evolution of the waiting time $W_n$ of customer $n$ can be described by the \emph{Lindley recursion} 
\begin{equation*}
W_{n+1}=(W_n+S_{n}-Z_{n+1})_+,\,\, n\in\N,
\end{equation*}
where, for simplicity, $W_0:=0$ is assumed (i.e.\ we start with an empty queue).






%
The ergodic theory of general state space Markov chains (see e.g.\ \cite{mt})
allows to treat the case where $(S_n)_{n\in\Z}$, $(Z_n)_{n\in\Z}$ are i.i.d.\ sequences, independent of
each other. In many situations, however,
there is non-zero correlation between inter-arrival times, 
so the Poisson assumption for the arrival process, which makes queueing models amenable to simple analysis, does
not apply. One also needs to consider more general single-server queues, such as G/GI/1 and G/G/1.
Dependence arises often in queuing networks where arrival processes are departure processes of other queues.
Processing operations like batching or the presence of multiple, different classes of customers can also
cause dependence, see \cite{heffes1980,heffes1986,sriram,fendick1989,fendick1991,kim}, consult also the overview
in Subsection 1.1 of \cite{whitt-you}. 

These studies show that the theoretical foundations 
of queuing systems should also be extended to
cases where the process $(S,Z)$ is merely stationary. The mathematics for such a setting is an order of magnitude more difficult
as $(W_{n})_{n\in\mathbb{N}}$ fails to be a Markovian process.
Henceforth, we assume the sequence 
$(S_{n},Z_{n})_{n\in\Z}\in\mathbb{R}_{+}^{2}$ to be (strong-sense) stationary.

As far as we know, Loynes was the first who studied the stability of waiting times in this general setting \cite{loynes1962} and introduced the terminology that a queue is \emph{subcritical} if $\E (S_0)<\E (Z_0)$, \emph{critical} if $\E (S_0)=\E (Z_0)$ and \emph{supercritical} if $\E (S_0) > \E (Z_0)$. Loynes proved for single-server queuing
systems that subcritical queues are stable, supercritical queues are unstable and
critical queues can be stable, properly substable, or unstable \cite{loynes1962}.
Stability of $W_n$, $n\in\N$ means here that there exist a unique limit distribution of $W_n$ as $n\to\infty$, 
whatever the initialization $W_{0}$ is. 

Gy\"orfi and Morvai extended Loynes' result in \cite{gyorfi2002} proving that for subcritical queues, a stronger version of stability called \emph{coupling} holds also true (see Theorem \ref{thm:gyorfi} below). 
We say that the sequence $(W_n)_{n\in\N}$ is \emph{coupled} with $(W_n')_{n\in\N}$ 
if this latter sequence is stationary and ergodic and there is an almost surely finite random variable $\tau$ such that $W_n = W_n' $ for $n>\tau$.
Gy\"orfi and Morvai's theorem concerning queues in this general setting reads
as follows:
\begin{theorem}\label{thm:gyorfi}
	Let $\xi_n = S_n-Z_{n+1}$ and assume that the process 
	$(\xi_n)_{n\in\Z}$ is ergodic with $\E (S_0)<\E (Z_0)$.
	Then $(W_n)_{n\in\N}$ is coupled with a stationary and ergodic $(W_n')_{n\in\Z}$ such that  $W_0'=\sup_{n\in\N} Y_n$, where
	$Y_0 = 0$ and $Y_n = \sum_{k=1}^n \xi_{-k}$, $n\ge 1$.
\end{theorem}

Little is known about the limit distribution and the speed of convergence. Actually, in 1961, Kingman published his famous approximate formula for the expectation of the stationary waiting time in G/G/1 queues \cite{kingman1961}. Kingman's approximation proved to be
very accurate when the system is operating close to saturation \cite{Harrison1993}.

In two consecutive papers \cite{AbateWhitt1994I,AbateWhitt1996II}, Abate, Choudhury and 
Whitt analyzed the tail probabilities of the steady-state waiting time, sojourn time and workload and suggested exponential approximations for them. They proved that in G/GI/1
queuing systems with service times, which are independent of an arbitrary stationary arrival process, asymptotic formulas for tail probabilities of the steady-state waiting time, sojourn time and workload are simply related \cite{AbateWhitt1996II}. 
 
For the speed of convergence, we found Theorem 4 on page 25 of 
\cite{bborovkov72} which gives the following upper bound
\begin{equation}\label{eq:borovkov}
	|\P (W_n\in B)-\P (W_0'\in B)|\le \P \left(\min_{0<k<n} X_k>\max (W_1,W_0'+\xi_0)\right),
\end{equation}
where $\xi_n$ is as in Theorem \ref{thm:gyorfi}, $(X_n)_{n\in\N}$ is defined as $X_0 = 0$, $X_n = \sum_{k=1}^{n} \xi_k$, $n\ge 1$
and $B\subset \mathbb{R}_{+}$ is an arbitrary Borel set. By taking the supremum in $B$ on the left hand side, 
we can easily obtain an estimate for the total variation distance of the corresponding distributions. 
However, the expression standing on the right hand side of \eqref{eq:borovkov} does not provide a concrete rate estimate.
For further information, the reader should consult the textbooks \cite{bborovkov} (in Russian) and \cite{borovkov} (in English).

The novelty of the present note is to regard the queuing systems under consideration as
Markov chains in random environment, i.e.\ Markov chains whose transition law is a stationary random process.
Such a framework allows to use the toolkit of Markov chain theory as presented e.g.\ in \cite{mt}:
Foster-Lyapunov and minorization conditions. We will rely on the recent advances made by \cite{attila} in the theory of
Markov chains in random environments. We will obtain rates for the convergence of the law of $W_{n}$
to its limit as $n\to\infty$ and a law of large numbers. Such results 
are anticipated on pages 503--504 of \cite{borovkov} but they are not
worked out in that book, neither elsewhere, as far as we know. 

If the inter-arrival times are i.i.d\ and the service times merely stationary \emph{and} these
two sequence are independent then the process $W$ is a Markov chain with driving noise $Z$ in the random environment
provided by $S$. Vice versa, if the service times are i.i.d.\ and the (independent from service) 
inter-arrival times stationary then $W$ is a Markov chain driven by $S$ in the random environment $Z$.
Hence both these special cases of queuing systems fit into the framework we propose.
More complex (e.g.\ multiserver) queuing systems 
could be analysed along similar lines but we do not pursue such ramifications here.
When both $S$ and $Z$ are merely stationary \emph{or} their independence fails then
the techniques of the present paper do not seem to be applicable.



Before expounding our new contributions, we introduce some more notations.
Throughout this paper we will be working on a probability space $(\Omega,\mathcal{F},P)$.
For a Polish space $\mathcal{X}$, we denote by $\B(\X)$ its Borel sigma-algebra. We denote by $\E[Z]$
the expectation of a real-valued random variable $Z$. For a $\mathcal{X}$-valued
random variable $X$ we will denote by $\law{X}$ its law on on $\B(\X)$.  
The set of probability measures on $\B(\mathcal{X})$ is denoted by $\M_1(\mathcal{X})$.
The total variation metric on $\M_1(\mathcal{X})$ is defined by 
	\begin{equation*}
	\dtv (\mu_1,\mu_2)=|\mu_1-\mu_2|(\X),\quad \mu_1,\mu_2\in\M_1(\mathcal{X}),
	\end{equation*}
	where $|\mu_1-\mu_2|$ denotes the total variation of the signed measure $\mu_1-\mu_2$. 
We do not indicate the dependence of the metric $\dtv$ on $\mathcal{X}$ since the latter
will always be clear from the context.
	
We now present our standing assumptions.
In a stable system service times should be shorter on average than inter-arrival times (i.e.\ we work in the
subcritical regime).
In our approach we also need that the service time sequence is independent of inter-arrival times.
So we formulate the following hypothesis.

\begin{assumption}\label{basic}
We stipulate
that $\E[S_0]<\E[Z_{0}]$ (where the latter may be infinity).
The sequences $(S_{n})_{n\in\mathbb{N}}$ and $(Z_{n})_{n\in\mathbb{N}}$ are independent.	
\end{assumption}

\begin{definition} We say that a sequence $(Y_n)_{n\in\N}$ of
real-valued random variables satisfies a \emph{G\"artner-Ellis-type condition} if
there is $\eta>0$ such that the limit 
		\begin{equation}\label{labbbe}
		\Gamma(\alpha):=\lim_{n\to\infty}\frac{1}{n}\ln \E\left[e^{\alpha(Y_1+\ldots+Y_n)}\right]
		\end{equation}
		exists for all $\alpha\in (-\eta,\eta)$ and $\Gamma$ is differentiable on $(-\eta,\eta)$. 
\end{definition}

\begin{remark}
{\rm The notion above is clearly inspired by the G\"artner-Ellis theorem, see Setion 2.3 of
\cite{dembo-zeitouni}, and it holds in a large class of
stochastic processes $Y$, well beyond the i.i.d.\ case. 
	For instance, let $Y_n=\phi(H_n)$ for a measurable 
	$\phi:\R^m\to\R$ satisfying a suitable growth condition and let $(H_{n})_{n\in\mathbb{N}}$ be an $\R^m$-valued  
	sufficiently regular Markov chain started from its invariant distribution. Then 
	\eqref{labbbe} holds true for all $\eta>0$, see Theorem 3.1 of \cite{KM2} for a precise formulation. 
See also Theorem 2.1 of \cite{dehling} for a non-Markovian example.

The use of G\"artner-Ellis-type conditions is not new in queuing theory: see e.g.\ 
Section 3 in \cite{gyorfi2002} and also \cite{gw},
where the exponential tail of the limit distribution of the queue length is studied 
when the arrivals are weakly dependent.
Here it is used to ensure (jointly with $\E[S_{0}]<\E[Z_{0}]$) the average contractivity of the system dynamics, see Assumptions
\ref{as:drift} and \ref{as:LT} below.}
	
\end{remark}			

We now recall a result on the ergodic behaviour of queuing systems with dependent
service times which was obtained in Theorem 4.7 of \cite{attila}.

\begin{theorem}\label{thm:queue} 
Let $(Z_{n})_{n\in\mathbb{N}}$ be an i.i.d.\ sequence and let $(S_{n})_{n\in\mathbb{N}}$
	be uniformly bounded, ergodic, satisfying a G\"artner-Ellis type condition. 
	Let us assume that $P(Z_{0}>z)>0$ for all $z>0$.
Then there exists a probability $\mu_{*}$ on $\mathcal{B}(\R_{+})$ such that 
		\begin{equation*}
		\dtv(\law{W_n},\mu_{*})\leq c_1\exp\left(-c_2n^{1/3}\right),
		\end{equation*}
		for some $c_1,c_2>0$. 
		Furthermore, for an arbitrary measurable and bounded $\Phi:\R_{+}\to\R$,
		\begin{equation}\label{torta}
		\frac{\Phi(W_1)+\ldots+\Phi(W_{n})}{n}\to \int_{\mathbb{R}_{+}} \Phi(z)\mu_{*}(\dint z),
		\end{equation}
		in probability.
		\hfill $\Box$
	\end{theorem}

In the present article we concentrate on the (arguably) more interesting case
where service times are independent but inter-arrival times may well be dependent.	
			
\begin{theorem}\label{thm:queue1} 
Let $(S_{n})_{n\in\mathbb{N}}$ be an i.i.d.\ sequence and let $(Z_{n})_{n\in\mathbb{N}}$
	be bounded, ergodic, satisfying a G\"artner-Ellis type condition. 
	Let us assume that $\E[e^{\beta_{0} S_{0}}]<\infty$ for some $\beta_{0}>0$ and $\law{S_{0}}$ has
	a density $s\to f(s)$ (w.r.t.\ the Lebesgue measure) which is bounded away from $0$ on 
	compact subsets of $\mathbb{R}_{+}$.
Then the conclusions of Theorem \ref{thm:queue} hold.
	\end{theorem}
	
\begin{remark}\label{bom}
{\rm The mathematical setting of Theorem \ref{thm:queue1} is significantly more involved that that of
Theorem \ref{thm:queue}. In Theorem \ref{thm:queue}, one may profit from the fact
that, freezing the values of the process $(S_{n})_{n\in\mathbb{N}}$, the waiting time becomes an inhomogeneous 
Markov chain with a particular state (the point $0$) which is a reachable atom.
In the proof of Theorem \ref{thm:queue1} one needs to guarantee ergodicity using 
a deeper coupling construction relying on the absolute continuity of the law of
$S_{0}$.}	
\end{remark}

The boundedness of $Z_{0}$ in Theorem \ref{thm:queue1} above is rather a stringent assumption though
it still covers a large class of models.
We can somewhat relax it in the next theorem, at the price of requiring more about $S_{0}$. Namely,
we assume an exponential-like tail for $S_{0}$ and for $Z_{0}$ a very light tail, like that 
of the Gumbel distribution at $-\infty$. It will become clear from the proof that requiring a thinner
tail for $S_{0}$ (e.g.\ Gaussian) would necessitate even more stringent tail assumptions
for $Z_{0}$. It would be nice to further relax these hypotheses in future work.
 
{}


	\begin{theorem}\label{thm:queue2} 
Let $(S_{n})_{n\in\mathbb{N}}$ be an i.i.d.\ sequence and let $(Z_{n})_{n\in\mathbb{N}}$ be ergodic,
	satisfying a G\"artner-Ellis type condition and 
	$P(Z_{0}\geq z)\leq C_{1}\exp\left(-C_{2}e^{C_{3}z}\right)${}
	with some $C_{1},C_{2},C_{3}>0$. 
	Let us assume that $\E[e^{\beta_{0} S_{0}}]<\infty$ for some $\beta_{0}>0$ and
	the law of $S_{0}$ has
	a nonincreasing density $s\to f(s)$ such that $f(s)\geq C_{4}e^{-C_{5}s}$, $s\geq 0$ for some $C_{4},C_{5}>0$.
Then there exists a probability $\mu_{*}$ on $\mathcal{B}(\R_{+})$ such that 
		\begin{equation*}
		\dtv(\law{W_n},\mu_{*})\leq c_1\exp\left(-c_2 n^{c_{3}}\right),
		\end{equation*}
holds for some $c_1,c_2,c_{3}>0$. 
Furthermore, for an arbitrary 
measurable and bounded $\Phi:\R_{+}\to\R$,{}
		\eqref{torta} holds in probability.
	\end{theorem}

Theorems \ref{thm:queue1} and \ref{thm:queue2} open the door for the statistical
analysis of such queuing systems. For example, choosing $\Phi(w):=1_{\{w\geq w_{0}\}}$, $w\in\mathbb{R}_{+}$
for some $w_{0}>0$, the above results guarantee that we can consistently estimate the probability of 
the waiting time exceeding $w_{0}$
in the stationary state by $(\Phi(W_1)+\ldots+\Phi(W_{n}))/{n}$ for large $n$.

In Section \ref{recall} below we will recall the notion of Markov chains in random environments and
certain results of \cite{attila} which we utilize. 
Sections \ref{p1} and \ref{p2} contain the proofs of
our two new theorems.

\begin{remark}{\rm The arguments of \cite{attila} also provide rate estimates
in the law of large numbers for functionals of the waiting time. Roughly speaking, if the 
process $Z$ has favourable enough mixing properties (such that its functionals satisfy the law of large
numbers with the usual $N^{-1/2}$ rate in $L^{p}$-norms)
then an estimate of $O(N^{-1/6})$ for the $L^{p}$-norms of the functional averages in \eqref{torta} can be obtained.
As it is not easy to provide a clear-cut set of conditions for suitable $Z$, we refrain from these ramifications here.}
\end{remark}
	
\section{Markov chains in random environments}\label{recall}

First we review the abstract setting of \cite{attila}.	
Let $\Y$, $\X$ be two Polish spaces  and $Y:\Z\times\Omega\to\Y$ a strongly stationary $\Y$-valued 
	stochastic process. 
	Let $Q:\Y\times\X\times\B(\X)\to [0,1]$ 
	be a mapping such that 
	\begin{enumerate}
		
		\item for all $B\in\B(\X)$ the function 
		$(y,x)\mapsto Q(y,x,B)$ is $\B(\Y)\otimes\B(\X)$-measurable and
		
		\item for all $(y,x)\in\Y\times\X$, 
		$B\mapsto Q(y,x,B)$ is a  probability measure on $\B(\X)$.
	
	\end{enumerate}
	We will consider $\X$-valued process $X_t$, $t\in\N$ such that $X_0=x_0\in\X$
	is fixed and for $t\in\N$,
	\begin{equation*}
		\P (X_{t+1}\in B\mid\sigma (X_s,\, 0\le s\le t;\, Y_s,\, s\in\mathbb{Z}))=Q(Y_t,X_t,B)\,\,\P-\Pas.
	\end{equation*}
	We interpret the process $X$ as a Markov chain in a random environment described by the process $Y$.
	
	\begin{definition}\label{def:act}
	Let $R:\X\times\B(\X)\to [0,1]$ be a probabilistic kernel. For
	a measurable function $\phi:\X\to\R_{+}$, we define
	\begin{equation*}
	[R\phi](x)=\int_\X \phi(z) R(x,\dint z),\,x\in\X.
	\end{equation*}
	\end{definition}
	Consistently with Definition \ref{def:act}, for $y\in\Y$, 
	$Q(y)\phi$ will refer to the action of the kernel $Q(y,\cdot,\cdot)$ on $\phi$.
	First, a Foster-Lyapunov-type drift condition is formulated.

	\begin{assumption}\label{as:drift}
		Let $V:\X\to\R_{+}$ be a measurable function.
	We consider measurable functions $K,\gamma:\mathcal{Y}\to \mathbb{R}_{+}$ with $K(\cdot)\geq 1$.
	We assume that, for all $x\in\X$ and $y\in\Y$,
	\begin{equation*}
	[Q(y)V](x)\le \ga (y)V(x)+K(y).
	\end{equation*}
	\end{assumption}

If $\ga,K$ are independent of $y$ and $\ga<1$ then this is the standard drift condition
for geometrically ergodic Markov chains, see Chapter 15 of \cite{mt}.
Here $\ga(y)\geq 1$ may well occur, but in the next assumption we stipulate that the system
dynamics, on long-time average, should be contracting.
	
	\begin{assumption}\label{as:LT}
		We require
		\begin{equation*}
			\bar{\ga}:=\limsup_{n\to\infty}\E^{1/n}\left(K(Y_0)\prod_{k=1}^{n}\ga (Y_k)\right) < 1.
		\end{equation*}
	\end{assumption}

We continue with familiar minorization conditions of
Markov chain theory, see Chapter 5 of \cite{mt}. In the jargon of that theory, we require
that there exist large enough ``small sets''.  

	\begin{assumption}\label{as:minor}
	Let $\ga (\cdot)$, $K(\cdot)$ be as in Assumption \ref{as:drift}.
	We assume that for some $0<\eps<1/\bar{\ga}^{1/2}-1$, there is a measurable function $\alpha:\Y\to (0,1)$ 
	and a probability kernel $\ka: \Y\times\B\to [0,1]$ such that, for all $y\in\Y$ and $A\in\B$,
	\begin{equation}\label{maudit}
		\inf_{x\in V^{-1}([0,R(y)])} Q(y,x,A)\ge (1-\alpha (y)) 
		\ka (y,A), \text{ where } R(y)=\frac{2K(y)}{\eps\ga (y)}	
	\end{equation}
	and $V^{-1}([0,R(y)])\neq\emptyset$.
	\end{assumption}
	
Finally, we also need to control the probability of $\alpha(Y_{0})$ approaching $1$. 
\begin{assumption}\label{as:myas}
$\lim_{n\to\infty} \E^{1/n^{\theta}}\left[\alpha (Y_0)^n\right] = 0$
holds for some $0<\theta<1$.
\end{assumption}
	
We now recall certain results of \cite{attila}: with the above presented assumptions, 
	the law of $X_n$ converges to a limiting law as $n\to\infty$, moreover, bounded functionals of
	the process $X$ show ergodic behavior provided that $Y$ is ergodic.
	\begin{theorem}\label{thm:TV}
		 Under Assumptions \ref{as:drift}, \ref{as:LT}, \ref{as:minor} and \ref{as:myas}, 
		 there exists a probability law $\mu_{\ast}$ such that 
		 \begin{align}
		 	\dtv (\law{X_{n}},\mu_{\ast}) \le C_{0}\left(e^{-\nu_{1}n^{2/3}}+\sum_{k=n}^{\infty}
		 		 	 \E^{\frac{\nu_{2}}{k^{\kappa\theta}}}[\alpha^{k^{\kappa}-1}(Y_0)]
		 	\right)\label{zoroaster}
		 \end{align}
	holds for all $n\in\N$, with some $C_{0},\nu_{1},\nu_{2}>0$, where $\kappa:=\frac{1}{3(1-\theta)}$.
If $Y$ is \emph{ergodic}, 
then for any bounded and measurable $\Phi:\X\to\R$ 
\begin{equation}\label{tortai}
\frac{\Phi (X_1)+\ldots + \Phi (X_n)}{n} \to \int_\X \Phi (z)\,\mu_{\ast}(\dint z), \,n\to\infty
\end{equation}
holds in probability.
	\end{theorem}
\begin{proof}
This follows from Theorems 2.8, 2.10 and Lemma 7.7 of \cite{attila}. We give a brief sketch of the main ideas nonetheless.
For simplicity, let $\alpha$ be constant. Let us start the process from two different initializations $X_{0},X_{0}'$.
The time interval $0,\ldots,T$ is cut into disjoint pieces of size $T^{2/3}$ (there are $T^{1/3}$
such pieces). Either the ``small sets'' of 
Assumption \ref{as:minor} are visited at least $T^{1/3}$ times or there is a whole piece of size $T^{2/3}$
without visits. In the first case Assumption \ref{as:minor} guarantees couplings of the two trajectories
outside a set of probability of the order $e^{-T^{1/3}}$ (since a minorization with constant $\alpha${}
can be used at each return to get couplings). In the other case the contractive effect of 
Assumption \ref{as:LT} acts for at least $T^{2/3}$ steps. As no returns to the small set occur during this
interval, the latter
probability is also small. An extension of this argument guarantees the Cauchy property for the
sequence $\mathcal{L}(X_{n})$ in total variation norm, which proves convergence, even with a rate.
The law of large numbers follows by a more refined construction in a similar vein.
\end{proof}

Our argument below consist in verifying that the queuing system in consideration satisfies	
the assumptions of Theorem \ref{thm:TV}.

\section{Proof in the unbounded case}\label{p1}

Throughout this section the assumptions of Theorem \ref{thm:queue2} are
in force. We will use results of the previous section in the setting $\mathcal{X}=\mathcal{Y}=\mathbb{R}_{+}$;
$Y_{n}:=Z_{n}$, $n\in\mathbb{N}$, i.e.\ the inter-arrival times will constitute the random environment
in which the waiting time evolves. We can easily extend the process $Y$ on the negative time axis in such a way that
$Y_{n}$, $n\in\mathbb{Z}$ is stationary. 
Define the parametrized kernel $Q$ as follows:
	\begin{equation*}
	Q(z,w,A):=\P \left[\left(w+S_{0}-z\right)_+\in A\right],\,
	z\in\mathcal{Y},\ w\in\mathcal{X},\,A\in\B\left(\mathcal{X}\right).
	\end{equation*}


	


We now turn to the verification of Assumptions \ref{as:drift} and \ref{as:LT}.
	
	\begin{lemma}\label{lem:queuing:DriftAndLT}	For some 
	$\bar{\beta}>0$ define
	\begin{align*}
	V(w) &:= e^{\bar{\beta}w}-1,\, w\ge 0,\\
	K(z) &:= \ga (z) :=e^{-\bar{\beta}z}\E\left[e^{\bar{\beta}S_{0}}\right],\,z\ge 0.
	\end{align*} 
	We have
	\begin{equation}\label{driff}
	[Q(z)V](w)\le \ga(z)V(w)+K(z)
	\end{equation}
	for all $z\in\mathcal{Y}$, $w\in\mathcal{X}$.
	Furthermore, choosing $\bar{\beta}$ small enough,
	\begin{equation}\label{control}
	\bar{\ga}:=\limsup_{n\to\infty}\E^{1/n}\left(K(Z_{0})\prod_{k=1}^{n}\ga (Z_k)\right) < 1.
	\end{equation}
	\end{lemma}
	\begin{proof}
	Let us estimate		
	\begin{eqnarray*}
	[Q(z)V](w) &=& \E[e^{\bar{\beta}(w+S_{0}-z)_{+}}]-1
				\le
				\E[e^{\bar{\beta}(w+S_{0}-z)}]\\
				&=& \ga(z)e^{\bar{\beta}w}
				= \ga(z)V(w)+ \ga(z),
	\end{eqnarray*}
			so \eqref{driff} holds.
Define 
\begin{equation*}
\la(\beta):=\Gamma(-\beta)+\ln(\E[e^{\beta S_{0}}]),\,\, \beta\in (-\bar{\eta},\bar{\eta}),
\end{equation*}
where $\bar{\eta}:=\min\{\eta,\beta_{0}/2\}$. Note that $Z_{0}$ has finite exponential moments of all
orders. Hence the functions 
			\begin{equation*}
			\la_n(\beta):=\frac{1}{n}\ln \E\left[e^{\beta\sum_{j=1}^n(S_{j}-Z_j)}\right],\ \beta\in (-\bar{\eta},
			\bar{\eta}),\ 
			n\geq 1	
			\end{equation*}
			are finite. They are also clearly convex. 
			Define 
			\begin{equation*}
				\psi_n(\beta):= \E\left[\frac{e^{\beta\sum_{j=1}^n(S_{j}-Z_j)}-1}{\beta}\right],\,\,
				\beta\in (0,\bar{\eta}),\,
				n\geq 1.
			\end{equation*}
			
			By the Lagrange mean value theorem and measurable selection, there exists a random variable 
			$\xi_n(\beta)\in [0,\beta]$ such that
			\begin{equation*}
			\psi_n(\beta)= \E\left[\left(\sum_{j=1}^n(S_{j}-Z_j)\right){}
			e^{\xi_n(\beta)\sum_{j=1}^n(S_{j}-Z_j)}\right].
			\end{equation*}
			Here 
			\begin{equation*}
			\left(\sum_{j=1}^n (S_{j}-Z_j)\right)e^{\xi_n(\beta)\sum_{j=1}^n(S_{j}-Z_j)}\leq \left(
			\sum_{j=1}^n S_{j}\right) e^{\bar{\eta} \sum_{j=1}^n S_{j}}
			\end{equation*}
			which is integrable. Hence reverse Fatou's lemma shows that
			\begin{equation*}
			\limsup_{\beta\to 0+}\psi_n(\beta)\leq \E\left[\sum_{j=1}^n (S_{j}-Z_j)\right]=n 
			\E\left[S_0-Z_0\right].
			\end{equation*}
			This implies that, for all $n\geq 1$, 
			$\la_n'(0)=\frac{1}{n}\lim_{\beta\to 0+}\psi_n(\beta)\leq{\E\left[S_0-Z_0\right]}<0$.
			
			Since $\la_n(\beta)\to\la(\beta)$ for $\beta\in (-\bar{\eta},\bar{\eta})$ by 
			the G\"artner-Ellis-type property of $(Z_{n})_{n\in\mathbb{N}}$,
			it follows from Theorem 25.7
			of \cite{rockafellar} that also $\la_n'(0)\to\la'(0)$ hence $\la'(0)<0$. 
			By Corollary 25.5.1 of \cite{rockafellar}, differentiability of $\la$ implies its
			\emph{continuous} differentiability, too. Hence
			from $\la(0)=0$ and $\la'(0)<0$ we obtain that there exists $\bar\beta>0$ satisfying 
			\begin{equation}\label{matra}
			\lim_{n\to\infty}\frac{1}{n}\ln \E e^{\bar{\beta}(S_1+\ldots+S_{n})-\bar{\beta}(Z_1+\ldots+Z_n)}<0.
			\end{equation}
			By \eqref{matra}, the long-time contractvity condition holds:
			\begin{equation}\label{juj}
			\limsup_{n\to\infty} \E^{1/n}[K(Z_{0})\ga(Z_1)\ldots\ga(Z_n)]<1
			\end{equation}
			since $K(z)\leq \E[e^{\beta_{0}S_{0}}]<\infty$ for all $z\in\mathcal{Y}$.	
	\end{proof}
	
Choose $\varepsilon:=(1/\bar{\ga}^{1/2}-1)/2$. 	
Notice that $R:=R(z):=2K(z)/(\varepsilon\gamma(z))=2/\varepsilon$ does not depend on $z$.
Now let us turn to the verification of the minorization condition.
Let $h:=(1/\bar{\beta})\ln\left(2/\varepsilon+1\right)$.
	
\begin{lemma}\label{lem:queuing:minor}  
		For $z\in\mathcal{Y}$, $A\in\B\left(\mathcal{X}\right)$ define 
		$\kappa(A):=\kappa(z,A):=\mathrm{Leb}(A\cap [h,h+1])$, where 
		$\mathrm{Leb}$ denotes the Lebesgue measure.
		Then 
		\begin{equation*}
		 \inf_{w\in V^{-1}([0,R])} Q(z,w,A)\ge (1-\alpha(z)) \kappa(A),	
		\end{equation*}
holds for $$
\alpha(z):=1-C_{4}e^{-C_{5}(z+h+1)}.
$$
\end{lemma}
\begin{proof} 
Notice that $V^{-1}([0,R])=[0,h]$.{}
For each $z\in\mathcal{Y}$, $A\in\mathcal{B}(\mathcal{X})$, $A\subset [h,h+1]$, and $w\in [0,h]$, 
\begin{align*}
		Q(z,w,A) &= \P \left(\left[w+S_{0}-z\right]_+\in A \right) \\
		&\ge \P \left(w+S_{0}-z\in A\right)           \\
	    &\ge \P\left(S_{0}\in (A+z-w)\right)   \\
	    &\ge f(z+h+1)\kappa(A)               \\
	    &\ge C_{4}e^{-C_{5}(z+h+1)}\kappa(A),
\end{align*}	
by translation invariance of the Lebesgue measure. We may conclude.
\end{proof}

\begin{proof}[Proof of Theorem \ref{thm:queue2}.]
Let $\mu_{0}$ denote the law of $Z_{0}$. Let $H>0$ be a constant to be chosen later.
From Lemma \ref{lem:queuing:minor},
\begin{align*}
	E[\alpha^{n}(Z_{0})] \leq \int_{0}^{\infty} \alpha^{n}(z)\mu_{0}(dz)   \le&\\
	\int_{0}^{H\ln(n)} (1- C_{4}e^{-C_{5}(z+h+1)})^{n}\mu_{0}(dz) + P(Z_{0}\geq H\ln(n))     \le&\\
	H\ln(n) e^{-nC_{4}e^{-C_{5}[H\ln(n)+h+1]}}+C_{1}\exp\left(-C_{2}e^{C_{3}H\ln(n)}\right) \le&\\ \label{astra} 
	H\ln(n)\exp\left\{-C_{6}n^{1-C_{5}H}\right\}+C_{1}\exp\left\{-C_{2}n^{C_{3}H}\right\}.
\end{align*}
with some constant $C_{6}$.
Choosing $H$ so small that $HC_{5}<1$ we get that Assumption \ref{as:myas} holds for $\theta$ small
enough. We can now invoke Theorem \ref{thm:TV}. The claimed
convergence rate also follows from \eqref{zoroaster}.
\end{proof}

\section{Proof in the bounded case}\label{p2}

Let the assumptions of Theorem \ref{thm:queue1} be in force.
Notice that Lemma \ref{lem:queuing:DriftAndLT} applies verbatim in this case, too.

\begin{lemma}\label{lem:queuing:minor1}  
	For $z\in\mathcal{Y}$, $A\in\B\left(\mathcal{X}\right)$ define 
	$\kappa(A):=\kappa(z,A):=\mathrm{Leb}(A\cap [h,h+1])$.
	Then $\inf_{w\in V^{-1}([0,R])} Q(z,w,A)\ge (1-\alpha) \kappa(A)$,
	holds for a constant $\alpha>0$.
\end{lemma}
\begin{proof} Let $M>0$ be such that $|Z_{0}|\leq M$ almost surely. As in Lemma \ref{lem:queuing:DriftAndLT},
\begin{align*}
Q(z,w,A) &\ge \P\left(S_{0}\in (A+z-w)\right)\\
         &\ge f(z+h+1)\kappa(A)\\
         &\ge [\inf_{v\in [0,M+h+1]}f(v)]\kappa(A),	
\end{align*}
which proves the statement since $f$ is bounded away from $0$ on compacts. In fact, one needs $f$ to be bounded
away from $0$ only on a \emph{fixed} compact which, however, depends on the laws of $S_{0}$, $Z_{0}$.
\end{proof}

\begin{proof}[Proof of Theorem \ref{thm:queue1}.] 
Assumption \ref{as:minor} holds by Lemma \ref{lem:queuing:minor1} and Assumption \ref{as:myas} holds trivially.
Lemma \ref{lem:queuing:DriftAndLT} implies Assumptions \ref{as:drift} and \ref{as:LT}. We can conclude
from Theorem \ref{thm:TV} with the choice $\theta=1/2$, noting that, in the present case,
\begin{equation*}
	\E^{\frac{\nu_{2}}{k^{1/3}}}[\alpha^{k^{2/3}-1}(Y_0)]]\leq
	\frac{1}{\alpha^{\nu_{2}}}\alpha^{\nu_{2}k^{1/3}}
\end{equation*}
and 
\begin{equation*}
	\sum_{k=n}^{\infty}\alpha^{\nu_{2}k^{1/3}}\leq C_{\sharp}\alpha^{\nu_{\sharp}k^{1/3}}
\end{equation*}
for some $C_{\sharp},\nu_{\sharp}>0$.
\end{proof}

\textbf{Aknowledgments.} Both authors thank for the support received from the ``Lend\"ulet'' 
grant LP 2015-6 of the Hungarian Academy of Sciences.

\end{document}